\documentclass{article}
\let\tldocenglish=1  
\usepackage{tex-live}
\usepackage[latin1]{inputenc} 
\usepackage[T1]{fontenc}

\usepackage[inactive]{srcltx}
\usepackage[centertags]{amsmath}
\usepackage{amsfonts}
\usepackage{amssymb}
\usepackage{amsthm}
\usepackage{newlfont}

\theoremstyle{definition}
\newtheorem{definition}{Definition}[section]
\theoremstyle{remark} \theoremstyle{theorem}
\newtheorem{remark}{Remark}[section]
\newtheorem{theorem}{Theorem}[section]
\newtheorem{lemma}{Lemma}[section]
\newtheorem{corollary}{Corollary}[section]
\newtheorem{example}{Example}[section]
\newtheorem{proposition}{Proposition}[section]

\title{Reconstruction of the core convex topology and its
applications in vector optimization and convex analysis}

\author{Ashkan Mohammadi\\[3mm]
Wayne State University,
Detroit, MI 48202;\\[1mm]
\url{ashkan.mohammadi@wayne.edu}\\[3mm] 
$\&$\\[3mm]
Majid Soleimani-damaneh\\[3mm]
University of Tehran, Tehran, Iran;\\[1mm]
 \url{soleimani@khayam.ut.ac.ir}}

\date{April 2017}

\begin{document}
\maketitle

\begin{abstract}
In this paper, the core convex topology on a real vector space $X$, which is constructed just by $X$ operators, is investigated. This topology, denoted by $\tau_c$, is the strongest topology which makes $X$ into a locally convex space. It is shown that some algebraic notions $(closure ~ and  ~ interior)$ existing in the literature come from this topology. In fact, it is  proved that algebraic interior and vectorial closure notions, considered in the literature as replacements of topological interior and topological closure, respectively, in vector spaces not necessarily equipped with a topology, are actually nothing else than the interior and closure with the respect to the core convex topology.
We reconstruct the core convex topology using an appropriate topological basis which enables us to characterize its open sets.

Furthermore, it is proved that $(X,\tau_c)$ is not metrizable when X is infinite-dimensional,
and also it enjoys the Hine-Borel property. Using these properties, $\tau_c$-compact sets are
characterized and a characterization of finite-dimensionality is provided. Finally, it is shown that the properties of the core convex topology lead to directly extending various important results in convex analysis and vector optimization from topological vector spaces to real vector spaces.\vspace{3mm}
\end{abstract}

\begin{multicols}{2}
\tableofcontents
\end{multicols}

\section{Introduction}
Convex Analysis and Vector Optimization under real vector
spaces, without any topology, have been studied by various
scholars in recent years
\cite{ada-2,ada-3,ada-4,bot-1,hol,jah,New2,kiy,zho1,zho}. Studying
these problems opens new connections between Optimization,
Functional Analysis, and Convex analysis. Since (relative)
interior and closure notions play important roles in many convex
analysis and optimization problems \cite{bao,jah,luc}, due to the
absence of topology, we have to use some algebraic concepts. To
this end, the concepts of \textit{algebraic (relative) interior}
and \textit{vectorial closure} have been investigated in the
literature, and many results have been provided invoking these
algebraic concepts; see e.g.
\cite{ada-2,ada-3,ada-4,fre,hol,jah,kiy,pen,pop,zho1,zho} and the
references therein. The main aim of this paper is to unifying vector optimization in real vector spaces with vector optimization in topological vector  spaces.

In this paper, core convex topology (see \cite{New1,New2}) on an arbitrary real vector
space, $X$, is dealt with. Core convex topology, denoted by $\tau_c$, is
the strongest topology which makes a real vector space into a
locally convex space (see \cite{New1,New2}). The topological dual of $X$ under
$\tau_c$ coincides with its algebraic dual \cite{New1,New2}. It is quite well known that when a locally convex space is given by a family of seminorms, the locally convex topology is deduced in a standard way and vice versa. In this paper, $\tau_c$ is reconstructed by a topological basis. It is known that algebraic
(relative) interior of a convex set is a topological notion
which can be derived from core convex topology \cite{New1,New2}.
We provide a formula for $\tau_c$-interior of an arbitrary (nonconvex) set with respect to the algebraic interior of its convex components. Furthermore, we show that vectorial closure is also a topological notion coming from core convex topology (under mild assumptions). According
to these facts, various important results, in convex analysis and
vector optimization can be extended easily from topological vector
spaces (TVSs) to real vector spaces. Some such results are
addressed in this paper. After providing some basic results about open sets in $\tau_c$, it is proved that, $X$ is not metrizable under $\tau_c$ topology if it is infinite-dimensional. Also, it is shown that
$(X,\tau_c)$ enjoys the Hine-Borel property.  A characterization of open sets in terms of there convex components is given.
Moreover, $\tau_c$-convergence as well as $\tau_c$-compactness are
characterized.

The rest of the paper unfolds as follows. Section 2 contains some
preliminaries and Section 3 is devoted to the core convex topology.
Section 4 concludes the paper by addressing some results existing
in vector optimization and convex analysis literature which can be
extended from TVSs to real vector spaces, utilizing the
results given in the present paper.

\section{Preliminaries}
Throughout this paper, $X$ is a real vector space, $A$ is a
subset of $X$, and $K \subseteq X$ is a nontrivial nonempty
ordering convex cone. $K$ is called pointed if $K \cap(-K) =
\{0\}$. $cone(A)$, $conv(A)$,
 and $aff(A)$ denote the cone generated by $A$, the convex hull of $A$, and
the affine hull of $A$, respectively.

For two sets $A,B\subseteq X$ and a vector $\bar{a}\in X$, we use
the following notations:
$$\begin{array}{c}
A\pm B:=\{a\pm b : ~ a\in A,~b\in B\},\vspace{2mm}\\
\bar{a}\pm A:=\{\bar{a}\pm a:~a\in A\},\vspace{2mm}\\
A \backslash B:=\{a\in A:~a\notin B\}.
\end{array}$$

$P(X)$ is the set of all subsets of $X$ and for $\Gamma \subseteq
P(X)$,
$$ \cup \Gamma := \{ x \in X: ~ \exists A \in \Gamma; ~ x \in A \} $$

The algebraic interior of $A\subseteq X$, denoted by $cor(A)$, and
the relative algebraic interior of $A$, denoted by $icr(A)$, are
defined as follows \cite{hol,New2}:
$$\begin{array}{c}
cor(A) := \{x \in A : ~\forall x^{'} \in X \;,\; \exists \lambda^{'} >0 ; ~~\forall \lambda \in [0,\lambda^{'}],~~ x+\lambda x^{'} \in A
 \},\vspace{2mm}\\
icr(A):= \{x \in A : ~\forall x^{'} \in L(A) \;,\; \exists \lambda^{'} >0; ~~\forall \lambda \in [0,\lambda^{'}],~~ x+\lambda x^{'} \in A
 \},
 \end{array}$$
where $L(A)=span(A-A)$ is the linear hull of $A-A$. When $cor(A)
\ne \emptyset$ we say that $A$ is solid; and we say that $A $ is
relatively solid if $icr(A) \ne\emptyset$. The set $A$ is called
algebraic open if $cor(A)=A$. The set of all elements of $X$ which do not belong to $cor(A)$ and
$cor(X  \setminus A)$ is called the algebraic boundary of $A$. The set $A$ is called algebraically bounded, if for every $ x \in A$ and every $ y \in X $ there is a $ \lambda  >  0 $ such that
$$     x + t y \notin A  \quad ~ \forall t \in [ \lambda , \infty ).        $$

If $A$ is convex, then there is a simple characterization of  $icr(A)$ as follows: $a\in icr(A)$ if and only if for each $x\in A$ there exists $\delta>0$ such that $a+\lambda(a-x)\in A,$ for all $ \lambda \in [0, \delta ].$

\begin{lemma}\label{C1}
Let $\{ e_{i} \}_{ i \in I}$ be a vector basis for $X$, and $A\subseteq X$ be nonempty and convex. $a\in cor(A)$ if and only if for each $ i \in I $ there exists scalar $\delta_{i}  > 0$ such that $a\pm \delta_{i} e_{i} \in A.$
\end{lemma}
 \begin{proof}
Assume that for each $ i \in I $ there exists scalar $\delta_{i}  > 0$ such that $a\pm \delta_{i} e_{i} \in A.$ Let $ d \in X$. There exist a finite set $J\subseteq I$ and  positive scalars $\lambda_{j}, \mu_{j}, \delta_{j},~j\in J$ such that
$$ d= \sum_{j \in J} \lambda_{j} e_{j} - \sum_{j \in J} \mu_{j} e_{j}  ,~~ a \pm \delta_{j} e_{j} \in A,  ~ \forall j \in J.$$
 Let $m:=Card(J)$ and  $\delta  >  0$. Considering $\delta\in (0,\min\{\frac{1}{2m},\frac{\delta_j}{2m\lambda_j}\})$, we have
$$ a + 2m \delta \lambda_{j} e_{j}  \in [ a ,  a+ \delta_{j} e_{j} ], ~  a - 2m \delta \mu_{j} e_{j}  \in [ a- \mu_{j} e_{j} ,  a], ~ \forall j \in J,$$
where $[x,y]$ stands for the line segment joining $x$,$y$. Since $ A $ is convex,
$$    a + 2m \delta \lambda_{j} e_{j}  \in A, ~    a - 2m \delta \mu_{j} e_{j}  \in A, ~ \forall j \in J,$$
and then, due to the convexity of $A$ again,
$$ \frac{a}{2}  + \sum_{j \in J} \delta \lambda_{j} e_{j}  \in \frac{A}{2},~~~ ~  \frac{a}{2}  - \sum_{j \in J}  \delta \mu_{j} e_{j}  \in  \frac{A}{2}.$$
This implies
$$ a +  \sum_{j \in J} \delta \lambda_{j} e_{j} - \sum_{j \in J} \delta \mu_{j} e_{j} \in A,$$
which means $ a + \delta d \in A.$ Furthermore, the convexity of $A$ guarantees that $ a + \lambda d \in A $ for all  $\lambda \in [0, \delta].$ Thus $a \in cor(A).$ The converse is obvious.
\end{proof}

Some basic properties of the algebraic interior are summarized in
the following lemmas. The proof of these lemmas can be found in
the literature; see e.g. \cite{ada-2,ada-3,New1,hol,New2}.

\begin{lemma} \label{p}
Let $A$ be a nonempty set in real vector space $X$. Then
the following propositions hold true:\\
1. If $A$ is convex, then $cor(cor(A))=cor(A)$,\\
2. $cor(A\cap B)=cor(A)\cap cor(B)$,\\
3. $cor(x+A)=x+cor(A)$ for each $x\in X,$ \\
4. $cor(\alpha A)=\alpha cor(A),$ for each $\alpha\in \mathbb{R} \setminus \{0\}  ,$\\
5. If $0\in cor(A)$, then $A$ is absorbing (i.e. $cone(A)=X$).
\end{lemma}

\begin{lemma} \label{00} Let $K \subseteq X$ be a convex cone. Then
the following propositions hold true:\\
i. If $cor(K)\ne \emptyset$, then $cor(K)\cup \{0\}$ is a convex cone,\\
ii. $cor(K)+K = cor(K)$,\\
iii. If $K,C \subseteq X$ are convex and relatively solid, then $icr(K)+icr(C)=icr(K+C).$\\
iv. If $ f:X \longrightarrow \mathbb{R}$  is a convex (concave) function, then $f$ is $\tau_{c}-$continuous.
\end{lemma}

Although the (relative) algebraic interior is usually defined in
vector spaces without topology, in some cases it might be
useful under TVSs too. It is because the algebraic (relative)
interior can be nonempty while (relative) interior is empty. The
algebraic (relative) interior preserves most of the properties of
(relative) interior.

Let $Y$ be a real topological vector space (TVS) with topology
$\tau$. We denote this space by $(Y,\tau).$ The interior of
$A\subseteq Y$ with respect to topology $\tau$ is denoted by
$int_{\tau}(A)$. A vector $a\in A$ is called a relative interior
point of $A$ if there exists some open set $U$ such that $U\cap
aff(A)\subseteq A.$ The set of relative interior points of $A$ is
denoted by $ri_{\tau}(A).$

\begin{lemma} \label{000}
Let $(Y,\tau)$ be a real topological vector space (TVS) and
$A\subseteq Y$. Then $int_{\tau}(A)\subseteq cor(A)$. If
furthermore $A$ is convex and $int_{\tau}(A)\neq \emptyset$,
then $int_{\tau}(A)=cor(A)$.
\end{lemma}

The algebraic dual of $X$ is denoted by $X^{'}$, and
$\langle.,.\rangle$ exhibits the duality pairing, i.e., for $l\in
X^{'}$ and $x\in X$ we have $\langle l,x\rangle:=l(x)$. The
nonnegative dual and the positive dual of $K$ are, respectively,
defined by
$$\begin{array}{c}
K^{+} := \{l \in X^{'}:~ \langle l,a\rangle \geq 0,~~ \forall a \in
K\},\vspace{2mm}\\
K^{+s} := \{l \in X^{'}:~ \langle l,a\rangle > 0,~~ \forall a
\in K \backslash\{0\}\}.
\end{array}$$

If $K$ is a convex cone with nonempty algebraic interior, then $
cor(K) = \{ x \in K: ~  \langle l,x\rangle > 0,~~ \forall l \in
K^{+} \setminus \{0\}  \}. $

The vectorial closure of $A$, which
is considered instead of closure in the absence of topology, is
defined by \cite{ada-2}
$$vcl(A):=\{b \in X : \exists x \in X \;;\; \forall \lambda^{'}
>0\;,\; \exists \lambda \in [0,\lambda^{'}]\;;\; b+\lambda x \in
A\}.$$

$A$ is called vectorially closed if $A=vcl(A)$. 

\section{Main results}

This section is devoted to constructing core convex topology
via a topological basis. Formerly, the core convex topology was constructed via a
family of separating semi-norms on $X$; see \cite{New2}. In this section, we are going to
construct core convex topology directly by characterizing
its open sets.
The first step in constructing a topology is defining its basis. The following
definition and two next lemmas concern this matter.

\begin{definition}\cite{mun}
Let $\digamma $ be a subset of $P(X)$, where $P(X)$ stands for
the power set of $X$. Then, $\digamma$ is called a topological
basis on $X$ if $X,\emptyset \in \digamma$ and moreover, the
intersection of each two members of $\digamma$ can be represented
as union of some members of $\digamma$.
\end{definition}

The following lemma shows how a topology is constructed from a
topological basis.

\begin{lemma} \label{topbasis}
If $\digamma$ is a topological basis on $X$, then the collection
of all possible unions of members of $\digamma$ is a topology on
$X$.
\end{lemma}

Lemma \ref{basis} provides the basis of the topology which we are
looking for. The proof of this lemma is clear according to Lemma \ref{p}.

\begin{lemma}\label{basis}
The collection
$$\mathfrak{B}:=\{A\subseteq X~:~cor(A)=A,~conv(A)=A\}$$
is a topological basis on $X$.
\end{lemma}


Now, we denote the topology generated by
$$\mathfrak{B}:=\{A\subseteq X~:~cor(A)=A,~conv(A)=A\}$$
by $\tau_c$; more precisely
$$\tau_c:=\{\cup\Gamma\in P(X)~:~\Gamma \subseteq  \mathfrak{B}\}.$$

The following theorem shows that $\tau_c$ is the strongest topology which
makes $X$ into a locally convex TVS. This theorem has been proved in \cite{New2} using a family of semi-norms defined on $X$. Here, we provide a different proof.

\begin{theorem}\label{111}~\\
i. $(X,\tau_c)$ is a locally convex TVS;\\
ii. $\tau_c$ is the strongest topology which makes $X$ into a locally
convex space.
\end{theorem}
\begin{proof} By Lemmas \ref{topbasis} and \ref{basis}, $\tau_c$ is a topology on $X$. \\
Proof of part i: To prove this part, we should show that
$(X,\tau_c)$ is a Hausdorff space, and two operators addition
$+:X\times X\rightarrow X$ and scalar multiplication
$+:\mathbb{R}\times X\rightarrow X$ are $\tau_c$-continuous.

\textit{Continuity of addition:} Let $x,y\in X$ and let $V$ be a
$\tau_c$-open set containing $x+y$. We should find two
$\tau_c$-open sets $V_{x}$ and $V_{y}$ containing $x$ and $y$,
respectively, such that $V_{x}+V_{y}\subseteq V$. Since
$\mathfrak{B}$ is a basis for $\tau_c$, there exists $A\in
\mathfrak{B}$ such that $$x+y\in A\subseteq V.$$ Defining
$$V_{x}:=\frac{1}{2} (A-x-y)+x~~\textmd{ and }~~V_{y}:=\frac{1}{2}(A-x-y)+y,$$
by Lemma \ref{p}, we conclude that $V_{x},V_{y} \in \mathfrak{B}$
and $V_{x} + V_{y}=A \subseteq V.$ Convexity of $A$, implies that
$V_x$ and $V_y$ are the desired $\tau_c$-open sets, and hence the
addition operator is $\tau_c$-continuous.

\textit{Continuity of scalar multiplication:} Let $x\in X$, $
\alpha \in \mathbb{R}$, and $V$ be a $\tau_c$-open set
containing $\alpha x$. without lose of generality, assume that $V
\in \mathfrak{B}$. We must show that there exist $\varepsilon >0$
and a $\tau_c$-open set $V_{x}$ containing $x$ such that
$$(\alpha - \varepsilon,\alpha + \varepsilon)V_{x} \subseteq V.$$
Since $\alpha x \in V = cor(V)$, by considering $d:=\pm x$ in the
definition of algebraic interior, there exists $\delta>0$ such
that
$$\alpha x+\lambda x \in V,~~\lambda \in (-\delta,\delta).$$
Define
$$U:=(V-\alpha x)\cap-( V- \alpha x).$$
We get $U=-U$, $0\in U$, and by Lemma \ref{p}, $U\in\mathfrak{B}$.
Furthermore, $U$ is balanced (i.e. $\alpha U\subseteq U$ for each
$\alpha\in [-1,1]$), because $U$ is convex and $0\in U$. Now, we
claim that
\begin{equation}\label{*}
(\alpha-\frac{\delta}{2},\alpha+\frac{\delta}{2})\bigg(\frac{1}{2\mid\alpha\mid+\delta}U+x\bigg)\subseteq
V.
\end{equation}

To prove (\ref{*}), let
$\alpha+t\in(\alpha-\frac{\delta}{2},\alpha+\frac{\delta}{2})$
with $\mid t\mid<\frac{\delta}{2}$. Therefore
$$\mid \frac{\alpha+t}{2\mid \alpha\mid +\delta}\mid  =\frac{1}{2}\mid  \frac{\alpha+t}{\mid \alpha\mid +\frac{\delta}{2}}\mid  \leq\frac{1}{2}
\frac{\mid \alpha\mid +\mid t\mid }{\mid \alpha\mid +\frac{\delta}{2}}
\leq\frac{1}{2}$$ Thus,
$$\frac{\alpha+t}{2\mid \alpha\mid +\delta}U\subseteq\frac{1}{2}U.$$
Hence,
$$(\alpha+t)(\frac{1}{2\mid \alpha\mid +\delta}U+x)=\frac{\alpha+t}{2\mid \alpha\mid +\delta}U+\alpha x+tx\subseteq\frac{1}{2}U+\alpha x+tx$$
$$\subseteq\frac{1}{2}(V-\alpha x)+\alpha x+tx=\frac{1}{2}V+\frac{1}{2}(\alpha x+2tx)\subseteq\frac{1}{2}V+\frac{1}{2}V=V.$$
This proves (\ref{*}). Setting $\varepsilon:=\frac{\delta}{2}$ and
$V_x:=\frac{1}{2\mid \alpha\mid +\delta}U+x$ proves the continuity of the scalar multiplication operator.

Now, we show that $(X,\tau_c)$ is a Hausdorff space. To this end,
suppose $x_{0}\in X\setminus\{0\}$. Consider $f\in X^{'}$ such
that $f(x_{0})=2$, and set $A:=\{x\in X~:~f(x)<1\}$ and
$B:=\{x~:~f(x)>1\}$. It is not difficult to see that
$A,B\in\mathfrak{B}$ and $x_0\in B$ while $0\in A$. This implies that $(X,\tau_c)$ is a Hausdorff space.\\
\\
ii. Let $\tau$ be an arbitrary topology on $X$ which makes $X$
into a locally convex space. Let $B_{\tau}$ be a locally convex basis of
topology $\tau.$ For each $U\in B_{\tau}$ we have
$cor(U)=int_{\tau}(U)=U$ (by Lemma \ref{000}) and hence $U\in
\mathfrak{B}$. Thus we have $B_{\tau}\subseteq \mathfrak{B}$,
which leads to $\tau \subseteq \tau_c$ and completes the
proof.\end{proof}

The interior of $A\subseteq X$ with respect to $\tau_c$ topology
is denoted by $int_c(A).$ The following theorem shows
that the algebraic interior (i.e. $cor$) for convex sets is a
topological interior coming from $\tau_c$.

\begin{theorem}\label{int}
\cite[Proposition 6.3.1]{New2} Let $A\subseteq X$ be a convex set. Then $$int_c(A)=cor(A).$$
\end{theorem}
\begin{proof} Since $(X,\tau_c)$ is a TVS, $int_cA\subseteq cor(A)$; see
Lemma \ref{000}. Since $A$ is convex, $cor(A)$ is also convex, and
furthermore $cor(cor(A))=cor(A)$ (by Lemma \ref{p}). Hence,
$cor(A)\in\mathfrak{B}$. Therefore, $cor(A)\subseteq int_c(A)$,
because $int_c(A)$ in the biggest subset of $A$ belonging to
$\mathfrak{B}$. Thus $int_c(A)=cor(A)$, and the proof is
completed.
\end{proof}

Notice that the convexity assumption in Theorem \ref{int} is
essential; see Example \ref{ex}.


The proof of the following result is similar to that of Theorem
\ref{int}.

\begin{theorem}\label{icr}
If $A$ is a convex  subset of $X$, then $icr(A)=ri_c(A)$, where
$ri_c(A)$ denotes the relative interior of $A$ with respect to
the topology $\tau_c$.
\end{theorem}

It is seen that the convexity assumption plays a vital role in Theorems \ref{int}  and \ref{icr}. In the following two results, we are going to characterize the $\tau_c$-interior of an arbitrary (nonconvex) nonempty set with respect to the $core$ of its convex components.  Since $int_{c} (A)\in \tau_{c}$ and $\mathfrak{B}$ is a basis for $\tau_{c},$ the set $int_{c} (A)$ could be written as union of some subsets of $A$ which are algebraic open.

\begin{lemma}
Let $A$ be a nonempty subset of real vector space $X.$ Then $A$ could be uniquely decomposed to the maximal convex subsets of $A$, i.e. $ A:= \bigcup_{i \in I} A_{i} $ where $A_{i}, ~i\in I$ are non-identical maximal convex subsets (not necessary disjoint) of $A$ (Here, $A_{i}$ sets are called convex components of $A$).
\end{lemma}
\begin{proof} For every $a \in A$, let $\Upsilon_a \subseteq P(A)$ be the set of all maximal convex subsets of $A$ containing $a$ (the nonemptiness of such $\Upsilon_{a}$ is derived from Zorn lemma). Set the index set $I:=  \bigcup_{a \in A}  \Upsilon_{a}$ (this type of defining $I$ enables us  to avoid repetition). For every $ i \in I,$ define $A_{i} = i.$ Hence, $ A= \bigcup_{i \in I} A_{i} $ where $A_{i}, ~i\in I$ are non-identical maximal convex subsets of $A.$ To prove the uniqueness of $A_{i}^{'}$s, suppose $ A= \bigcup_{j \in J} B_j$ such that $B_j~(j \in J)$ are non-identical maximal convex subsets of $A.$ Let $j \in J$ and $ x \in B_j $; then $ B_{j} \in \Upsilon_{x}$, and hence there exists $i \in I$ such that $ B_{j} = i =A_i.$ This means $\{ B_j : ~ j \in J  \} = \{ A_i : ~ i \in I \}$, and the proof is completed.\end{proof}

\begin{theorem}
Let $A$ be a nonempty subset of real vector space $X$. Then
$$ int_c (A) = \bigcup_{i \in I} cor(A_i),$$
where $A_i, ~ i \in I$ are convex components of $A$

\end{theorem}
\begin{proof} Let $ a \in  int_c (A).$ There exists $ V \in \mathfrak{B}$ such that $ a \in V \subseteq A.$ Set $ \Pi := \{ B \subseteq A : ~  V \subseteq B , ~ B=conv (B) \}.$ Clearly $V \in \Pi$ and hence $\Pi \neq \emptyset$. Furthermore it is easy to verify that each chain (totaly ordered subset) in $\Pi$ has an upper bound within $\Pi$. Therefore, using Zorn lemma, $\Pi$ has a maximal element. Let $B_{*}$ be maximal element of $\Pi$. Obviously, $B_{*}$ is a convex component of $A$. It leads to the existence of an $i \in I$ such that $a \in V \subseteq B_{*} = A_i $. Therefore, $ a \in cor(A_{i}).$ To prove the other side, suppose $ a \in cor(A_i)$ for some $ i \in I.$ Since $A_i$ is convex, by Theorem \ref{int}, $a \in cor(A_i) = int_c (A_i) \subseteq int_c (A).$\end{proof}


\begin{example}\label{ex}
Consider
$$ A:=\{ (x,y) \in \mathbb{R}^2 :   y \geq x^2 \} \cup \{ (x,y) \in \mathbb{R}^2 : ~  y \leq -x^2 \} \cup  \{ (x,y) \in \mathbb{R}^2 : ~  y=0  \}  $$
as a subset of $\mathbb{R}^2.$ It can be seen that $ (0,0) \in cor(A),$ while $(0,0) \notin int_c (A) = int_{ \|.\|_{2} } (A).$ However,  $int_c (A) = \bigcup_{i=1}^{4} cor(A_{i}) = \bigcup_{i=1}^{4} int_{ \|.\|_{2} } (A_{i}),$ where $A_{i}, ~ i=1,2,3,4$ are convex components of $A$ as follows
$$ A_1 = \{ (x,y) \in \mathbb{R}^2 :   y \geq x^2 \},  ~ A_2 = \{ (x,y) \in \mathbb{R}^2 : ~  y \leq -x^2 \}, $$
$$ A_3 = \{ (x,y) \in \mathbb{R}^2 : ~  y=0  \} , ~ A_4 = \{ (x,y) \in \mathbb{R}^2 : ~  x=0  \}. $$ \qed

\end{example}


The following theorem shows that the topological dual of
$(X,\tau_c)$ is the algebraic dual of $X$. In the proof of this
theorem, we use the topology which $X^{'}$ induces on $X$. This
topology, denoted by $\tau_0$, is as follows:
$$\tau_0=\{\cup \Gamma \in P(X)~:~\Gamma \subseteq \Psi\},$$
where
$$\Psi=\{ A \subseteq X~:~ A=f_1^{-1}(I_1)  \cap f_2^{-1}(I_2) \cap... \cap f_n^{-1}(I_n) \textmd{ for some } n\in \mathbb{N}, \textmd{ some}$$
$$~~~~~~~~~~\textmd{open intervals }I_1, I_2, ... , I_n\subseteq \mathbb{R}\textmd{
and some }f_1, f_2, ..., f_n \in X^{'}\}.$$

\begin{theorem}\label{int1} \cite{New1} $(X,\tau_c)^*=X^{'}.$
\end{theorem}
\begin{proof} Let $\tau_0$ denote the topology which $X^{'}$ induces
on $X$. By \cite[Theorem 3.10]{rud}, $\tau_0$ makes $X$ into a
locally convex TVS and $(X,\tau_0)^*=X^{'}$. By Theorem \ref{111},
$\tau_0\subseteq \tau_c$. Hence $(X,\tau_c)^*=X^{'}$.\end{proof}

The following result provides a characterization of
finite-dimensional spaces utilizing $\tau_c$ and the
topology induced by $X^{'}$ on $X$.

\begin{theorem}
$X$ is finite-dimensional if and only if $\tau_0=\tau_c.$
\end{theorem}
\begin{proof} Assume that $X$ is finite-dimensional. Since there is only one
topology on $X$ which makes this space a TVS, we have
$\tau_0=\tau_c.$

To prove the converse, by indirect proof assume that
$dim(X)=\infty$. Let $\beta=\{x_{i}\}_{i\in I},$ be an ordered
basis of $X$; and $[x]_{\beta}$ denote the vector of coordinates
of $x\in X$ with respect to the basis $\beta$. It is easy to show
that, $[x]_{\beta} \in \mathit{l}_{1}(I)$ for each $x\in X$, where
$$l_1(I)= \{\{t_{i}\}_{i\in I}\subseteq \mathbb{R}~:~\|\{t_{i}\}\|_{1} = \sum_{i \in I}\mid t_{i} \mid  <  \infty  \}.$$
Define $\Omega: X \longrightarrow \mathbb{R}$ by $\Omega
(x)=\left \| [x]_{\beta} \right \|_{1}. $ Since $ \left \| .
\right \|_{1} $ is a norm  on $ \mathit{l}_{1}(I) $, thus $
\Omega $ is also a norm on $X.$ We denote this norm by
$\|.\|_{0}$.

Since $\tau_c$ is the strongest locally convex topology on $X$,
we have $B_1^{\|.\|_{0}}(0)\in \tau_c $, where
$B_1^{\|.\|_{0}}(0)$ stands for the unit ball with $\|.\|_0.$  On
the other hand, every $\tau_0$-open set containing origin,
contains an infinite-dimensional subspace of $X$ (see
\cite{rud}). Hence, $ B_1^{\|.\|_0}(0)\notin \tau_0.$ This
implies $\tau_0\neq \tau_c$, and the proof is completed.
\end{proof}

Theorem \ref{vcl} demonstrates that the vectorial closure ($vcl$)
for relatively solid convex sets, in vector space $X$, is a
topological closure coming from $\tau_c$. The
closure of $A\subseteq X$ with respect to $\tau_c$ is denoted by $cl_c(A).$

\begin{theorem} \label{vcl}
Let $A\subseteq X$ be a convex and relatively solid set. Then $vcl(A)=cl_c(A).$
\end{theorem}
\begin{proof} Since  $(X,\tau_c)$ is a TVS, it is easy to verify that
$vcl(A)\subseteq cl_c(A).$ To prove the converse, let  $x\in
cl_c(A)$ and $a\in icr(A)$. Without loss of generality, we assume
$a=0$, and then we have $Y:= span(A)= aff(A)=L(A),$ and $0\in
cor^{Y}(A)$, where $cor^Y(A)$ stands for the algebraic interior of
$A$ with respect to the subspace $Y$. We claim that $x\in Y.$ To
show this, assume that $x\notin Y$; then there exists a functional
$f\in X^{'}$ such that $f(x)=1$ and $f(y)=0$ for each $y\in Y$.
Therefore the set $$U:= \{ z \in X:~ f(z)>\frac{1}{2}\}$$ is a
$\tau_c-$open neighborhood of $x$  with $U\cap A=\emptyset,$
which contradicts $x\in cl_c(A).$ Now, we restrict our attention
to subspace $Y$. By Theorem \ref{int}, $a\in int^Y_c(A)$, and
thus, by \cite[Lemma 1.32]{jah}, $(x,a]\subseteq {int}_c^Y(A) \subseteq A $,
which means $x\in vcl(A)$. Therefore  $cl_c(A)\subseteq vcl(A)$,
and the proof is completed.
\end{proof}

\begin{corollary}
Let $A$ be a nonempty subset of $X$ with finite many convex components.  Then
$$  cl_c (A) = \bigcup_{i=1}^{n} vcl (A_{i})     $$
where $ A_{i} , ~ i=1,2,...,n $ are convex components of $A$.
\end{corollary}
\begin{proof} According to Theorem \ref{vcl}, we have
$$ cl_c (A) = \bigcup_{i=1}^{n} cl_c (A_{i}) = \bigcup_{i=1}^{n} vcl (A_{i}).$$\end{proof}

The equality $ cl_c(A) = \bigcup_{i \in I } vcl(A_{i}) $ may not be true in general; even if each convex component, $A_{i}$, is relatively solid. For example, consider $\mathbb{Q}$ as the set of rational numbers in $\mathbb{R}$. The convex components of $\mathbb{Q}$ are singleton, which are relatively solid. Therefore $\mathbb{Q}=\bigcup_{q \in Q}\{q\}$, while
$$ \bigcup_{q \in\mathbb{Q}} vcl(\{q\}) = \bigcup_{q \in \mathbb{Q}} \{q\} = \mathbb{Q} \neq \mathbb{R} = cl_c(\mathbb{Q}).$$


A TVS $(X,\tau)$ is called metrizable  if there exists a metric
$d:X\times X\longrightarrow \mathbb{R}$ such that $d$-open sets
in $X$ coincide with $\tau$-open sets in $X$. Now, we are going
to show that $(X,\tau_c)$ is not metrizable when $X$ is
infinite-dimensional. Lemma \ref{lid} helps us to prove it. This
lemma shows that every algebraic basis of the real vector
space $X$ is far from the origin with respect to $\tau_c$
topology.

\begin{lemma}\label{lid}
Let $\{x_{i}\}_{i\in I}$ be an algebraic basis of $X$. Then
$0\notin cl_c(\{x_{i}\}_{i\in I}).$
\end{lemma}
\begin{proof} Define\\ $$\begin{array}{l}
A:=\bigg\{\sum_{i\in F}\lambda_{i}x_{i}-\sum_{i\in F}\mu_{i}x_{i}:\sum_{i\in F}(\lambda_{i}+\mu_{i})
=1,~\lambda_{i},\mu_{i}>0;~\forall i\in F,\\
~\textmd{~~~~~~~~~~~~~~~~~~~~~~~~~~~~~~~~~~~~~~~~~~~~~~~~~~~~~~~~~~
and}~F\textmd{ is a}\textmd{finite subset of }I\bigg\}.
\end{array}$$
We claim that the following propositions hold true; \\
a. $A~is~convex $, \\
b. $ 0 \in cor(A) $, \\
c. $\{x_{i}\}_{i\in I}\cap A=\emptyset$.

To prove (a), assume that $x,y\in A$ and $t\in(0,1).$ Then there
exist positive scalars
$\lambda_{i},\mu_{i},\overline{\lambda}_{i},\overline{\mu}_{i}$
and finite sets $F,S\subseteq I$ such that;

$$x=\sum_{i \in F} \lambda_{i} x_{i} - \sum_{i \in F} \mu _{i} x_{i}, ~~~~~\quad y=\sum_{i \in S} \overline{\lambda}_{i} x_{i}-\sum_{i \in S} \overline{\mu}_{i}x_{i},$$
and \quad ~~\quad \quad $\sum_{i \in F} (\lambda_{i} + \mu_{i}) =  \sum_{i \in S}(\overline{\lambda}_{i} + \overline{\mu}_{i})=1$.\\\\
Hence,
$$ tx + (1-t) y = \sum_{i \in F \cup S }   \widehat{\lambda}_{i}  x_{i} - \sum_{i \in F \cup S } \widehat{\mu}_{i}x_{i},~~~~~~~~~~~~$$
 where
 $$\widehat{\lambda}_{i} = \left\{\begin{array}{ll}
t \lambda_{i} + (1-t)\overline{\lambda}_{i} &  ~~~i \in F \cap S \\
t \lambda_{i} & ~~~i \in F \setminus S \\
(1-t)\overline{\lambda}_{i} &  ~~~i \in S \setminus F
\end{array}\right.~~~~~~\widehat{\mu}_{i}=\left\{\begin{array}{ll}
t \mu_{i} + (1-t)\overline{\mu}_{i} & ~~~ i \in F \cap S \\
t \mu_{i} & ~~~i \in F \setminus S \\
(1-t)\overline{\mu}_{i} &  ~~~i \in S \setminus F
\end{array}\right.$$
Furthermore, $$\begin{array}{ll} \sum_{i\in F\cup
S}\widehat{\lambda}_{i}+\widehat{\mu}_{i}&=\sum_{i \in F \cap S}
\left [t(\lambda_{i} + \mu_{i} ) + (1-t) (
\overline{\lambda}_{i}+\overline{\mu}_{i})\right]+\sum_{i\in F
\setminus S}t(\lambda_{i}+\mu_{i})\vspace{2mm}\\
&+\sum_{i \in S \setminus F } (1-t)
(\overline{\lambda}_{i}+\overline{\mu}_{i})=\sum_{i \in F } t
(\lambda_{i}+\mu_{i})+\sum_{i \in
S}(1-t)(\overline{\lambda}_{i}+\overline{\mu}_{i})\vspace{2mm}\\
&=t+(1-t)=1. \end{array}$$ Therefore $tx+(1-t)y\in A$, and hence
$A$ is a convex set.

To prove (b), first notice that $\pm \frac{1}{2} x_{i} \in A $ for each $i \in I$; then due to the convexity of $A$, from (a) we get $0\in cor(A)$, because of Lemma \ref{C1}.

We prove assertion (c) by indirect proof. If $\{x_{i}\}_{i\in
I}\cap A \neq \emptyset$, then $$x_{j} = \sum_{i \in F}
\lambda_{i} x_{i} - \sum_{i \in F} \mu _{i} x_{i}$$ for some $ j
\in I $, and some finite set $F\subseteq I$. Also, $\lambda_i$
and $\mu_i$ values are positive and $\sum_{i\in
F}(\lambda_i+\mu_i)=1$. Since $\{x_{i}\}_{i\in F}\cup \{ x_{j}\} $
is a linear independent set, we have $j\in F$, and also $ 1 =
\lambda_{j} - \mu_{j} $. Hence $ \lambda_{j}
> 1$ which is a contradiction. This completes the proof of
assertion (c).

Now, we prove the lemma invoking $(a)-(c)$. Since $A$ is convex,
by theorem \ref{int} and claim (b), $cor(A)$ is a $\tau_c-$open
neighborhood of $0$. On the other hand, (c) implies $cor(A)\cap
\{ x_{i} \}_{ i \in I } = \emptyset.$ Thus, $0\notin cl_c( \{
x_{i} \}_{ i \in I } ).$
\end{proof}


\begin{theorem} If $X$ is infinite-dimensional, then $(X,\tau_c )$
is not metrizable.
\end{theorem}
\begin{proof} Suppose that $(X,\tau_c)$ is an infinite-dimensional metrizable TVS with metric $d(.,.).$ Then for every $n\in
\mathbb{N}$ choose $x_{k}\in X,~k=1,2,\ldots,n$ such that
$d(0,x_{k})<\frac{1}{k}$ and $\{x_{1},x_{2},\ldots,x_{n}\}$ is
linear independent. This process generates a linear independent
sequence $\{x_{n}\}_{n\in\mathbb{N}}$ such that,
$x_{n}\longrightarrow^{^{\hspace{-4mm}\tau_c}} 0$ (i.e. $\{x_{n}\}$ is
$\tau_c-$convergent to zero). Furthermore, this sequence can be
extended to a basis of $X$, say $\{x_{i}\}_{i\in I}$. This makes
a contradiction, according to Lemma \ref{lid}, because $0\in
cl_c(\{x_{i}\}_{i\in I}).$
\end{proof}

In every topological vector space, compact sets are closed and
bounded, while the converse is not necessarily true. A topological
vector space in which every closed and bounded set is
compact, is called a \textit{Hiene-Borel} space. Although it is
almost rare that an infinite-dimensional locally convex space be
a Hine-Borel space (for example, normed spaces), the following
result proves that $(X,\tau_c)$ is a Hine-Borel space.

\begin{theorem}\label{hine}
$(X,\tau_c)$ enjoys the Hine-Borel property. Moreover, every
$\tau_c-$compact set in $X$ lies in some finite-dimensional
subspace of $X.$
\end{theorem}
\begin{proof} Let $A$ be a $\tau_c$-closed and $\tau_c$-bounded subset of
$X$. First we claim that the linear space $\langle A\rangle$,
i.e. the smallest linear subspace containing $A$, is
finite-dimensional. To see this, by indirect proof, assume that
$\langle A\rangle$ is an infinite-dimensional subspace of $X$;
then there exists a linear independent sequence
$\{x_{n}\}_{n=1}^{\infty}$  in $A$. Thus, the sequence $\{\frac{
x_{n}}{n}\}_{n=1}^{\infty}$ is a linear independent set in $X,$
and also $\{x_{n}\}_{n=1}^{\infty}$ is $\tau_c$-bounded. Hence,
$\frac{x_{n}}{n}\longrightarrow^{^{\hspace{-4mm}\tau_c}}0$.
Furthermore, $X$ has a basis $\{x_{i}\}_{i\in I}$, containing
$\{\frac{x_{n}}{n}\}_{n=1}^{\infty}.$ This makes a
contradiction, according to Lemma \ref{lid}, because $0\in
cl_c(\{x_{i}\}_{i\in I}).$  Hence, $A$ is a closed and bounded
set contained in a finite-dimensional space $\langle A\rangle$.
Therefore, by traditional Hine-Borel theorem in
finite-dimensional spaces, $A$ is a compact set in $\langle
A\rangle$, and hence $A$ is a $\tau_c$-compact set in $X.$
\end{proof}


One of the important methods to realize the behavior of a
topology defined on a nonempty set is to perceive (characterize)
the convergent sequences. Assume that a sequence
$\{x_n\}_{n=1}^{\infty}$ is $\tau_c$-convergent to some $x\in
X$ (i.e, $x_{n}\longrightarrow^{^{\hspace{-4mm}\tau_c}} x$). Thus,
$\{x_n\}_{n=1}^{\infty}\cup\{x\}$ is a $\tau_c$-compact set.
Therefore, by Theorem \ref{hine},
$\{x_{n}\}_{n=1}^{\infty}\cup\{x\}$ lies in a finite-dimensional
subspace of $X$. This shows that, convergent sequences of
$(X,\tau_c)$ are exactly norm-convergent sequences contained in
finite-dimensional subspaces of $X$. Hence, every sequence with
infinite-dimensional $span$ could not be convergent. So, the
convergence in $(X,\tau_c)$ is almost strict; however, this is
not surprising because \textit{strongest-topology} ($\tau_c$) contains more number
of open sets than any other topology which makes $X$ into a locally
convex space.


\section{Applications}

In this section, we address some important results in convex
analysis and optimization under topological vector spaces which
can be directly extended to real vector spaces, utilizing
the main results presented in the previous section. Some
of these results can be found in the literature with different
complicated proofs. Subsection
4.1 is devoted to some Hahn-Banach type separation theorems.

\subsection{Separation}

\begin{theorem} \label{sep1}
Assume that $A,B$ are two disjoint convex subsets of $X$, and
$cor(A)\neq \emptyset.$ Then there exist some $f\in
X^{'}\setminus \{0\}$ and some scalar $\alpha \in \mathbb{R}$
such that
\begin{equation}\label{0A1} f(a)\leq \alpha \leq  f(b),~~~\forall a\in A, b\in B.\end{equation}
Furthermore,
\begin{equation}\label{0B1} f(a)<\alpha,~\forall a \in cor(A).\end{equation}
\end{theorem}
\begin{proof} Since $(X,\tau_c)$ is a TVS and $int_c(A)=cor(A)\neq
\emptyset$, by a standard separation theorem on topological
vector spaces (see \cite{rud}), there exists some
$f\in(X,\tau_c)^*$ and some scalar $\alpha\in\mathbb{R}$
satisfying (\ref{0A1}) and (\ref{0B1}). This completes the proof
according to Theorem \ref{int1}.\end{proof}

\begin{theorem}
Two disjoint convex sets $A,B\subseteq X$ are strongly separated
by some $f\in X^{'}$ if and only if there exists a convex
absorbing set $V$ in $X$ such that $(A+V)\cap B=\emptyset$.
\end{theorem}
\begin{proof} It is easy to check that $V$ and then $A+V$ are convex solid
sets. Now it is enough to apply Theorem \ref{sep1} for two sets
$A+V$ and $B$.
\end{proof}

In the following theorem, we brought some new conditions on two
convex sets that guarantee their strongly separation by some
$f\in X^{'}$.

\begin{theorem}\label{sep2}
Suppose that $A,B\subseteq X$ are two disjoint convex sets such
that, $A$ is vectorial closed and relatively solid. If furthermore,
$B\subseteq Y$, where $Y$ is a finite-dimensional subspace of
$X$, and $B$ is compact in $Y$ (note that $Y$ is homeomorphic to
$ \mathbb{R}^{dim Y}$), then $A$ and $B$ can be strongly
separated by some $f\in X^{'}$.
\end{theorem}
\begin{proof} Since $A$ is convex and relatively solid then, by Theorem \ref{vcl} we have
$A=vcl(A)=cl_c(A).$ Thus $A$ is $\tau_c$-closed. Let
$\{V_{i}\}_{i\in I}$ be a $\tau_c$-open cover of $B$ in $X$.
Then $\{V_{i}\cap Y\}_{i \in I} $ is a norm-open cover of $B$ in
$Y$. Since $B$ is compact in $Y$, $ \{ V_{i} \cap Y \}_{i \in I }
$ and then $\{ V_{i} \}_{i \in I }$ admit finite subcovers of
$B$. Hence, $B$ is $\tau_c-$compact in $X$. Now for completing
the proof, it is enough to apply the classic strong separation
theorem for closed convex set $A$ and convex compact set $B$ in
locally convex space $(X,\tau_c)$; see  \cite{rud}.
\end{proof}

\begin{theorem}
(Separation of cones). Let $M,K$ be two solid convex cones in $X$.
If $M\cap cor(K)=\emptyset$, then there exists a functional $f\in
X^{'} \backslash \{0\}$ such that,
$$f(m)\leq 0 \leq f(k)\qquad\quad\forall (k\in K , m \in M),$$
and furthermore,
$$f(k)> 0 \qquad\quad \forall k \in cor(K),$$
and
$$f(m)<0 \qquad\quad \forall m \in cor(M).$$
\end{theorem}
\begin{proof} This theorem results from Theorems \ref{111}, \ref{int}, and
\ref{int1} in the present paper and standard separation theorem
in TVSs (see \cite{rud}).
\end{proof}


The following result is a Sandwich Theorem between two functions in real vector spaces without topology.

\begin{theorem} \label{san} \textit{(Sandwich Theorem)}
Let $X$ be a real vector space and let $f,g : X \longrightarrow \mathbb{R}$ be convex and concave functions, respectively, such that $ g(x) \leq f(x)$ for each $x \in X $. Then there exist some $l\in X^{'}$ and $\alpha \in \mathbb{R}$ such that
$$g(x) \leq l(x) +  \alpha \leq  f(x),~~\forall x \in X.$$
\end{theorem}
\begin{proof} Both functions $f,g$ are $\tau_c$-continuous by Lemma \ref{00}. Furthermore, these functions satisfy all assumptions of the Sandwich Theorem under TVS $(X, \tau_c)$ (see \cite[Page 14]{ws}). This completes the proof.\end{proof}


\begin{proposition}
Let $A$ be a convex and relatively solid set in $X$. Then \\
1. $ cor (A^c )  = int_c ( A^c ) $,\\
2. $\delta  A = vcl(A) \setminus cor(A)$, where $ \delta  A $ stands for the algebraic boundary of $A$.
\end{proposition}
\begin{proof} 1. It is enough to prove $cor (A^c )  \subseteq int_c ( A^c ). $ To see this, suppose $x \in cor ( A^c ).$ We claim that $ x \notin vcl(A);$ otherwise there exist $d \in X$ and $ \overline{ \lambda } > 0 $ such that $ x + \lambda d \in A $ for each $\lambda  \in [0, \overline{ \lambda } ]$ which contradicts $x\in cor ( A^c ).$ Therefore, $ x \notin vcl(A).$ By Theorem \ref{sep2} in the present paper, there exist $ f \in X^{'}  \setminus \{0\} $  and  $ \alpha  >  0 $ such that
$$    f(a)  < \alpha  <  f(x)  \quad \quad \forall a \in A.          $$
Therefore $ U:= \{ z \in X :~ \alpha  <  f(z)  \}  $ is a convex and $ \tau_c -$open neighborhood of $ x $ satisfying $U \subseteq A^c .$ So $ x \in int_c ( A^c ). $

2.  According to definition of algebraic boundary of $A$, we have $ \delta A = (  cor (A^c) \cup cor(A) )^c$. Using Part 1, we obtain
$$ \delta A = ( int_c (A^c) \cup  int_c (A) )^c = [ int_c (A^c) ]^c \cap  [ int_c (A) ]^c  = cl_c (A)  \setminus int_c (A). $$
Now, using Theorems \ref{int} and \ref{vcl}, we conclude $\delta  A = vcl(A) \setminus cor(A). $
\end{proof}

We close this subsection by a nonconvex separation theorem in real
vector spaces. This theorem is a direct consequence of a
nonconvex separation theorem presented in \cite{ger}.

\begin{definition}
Let $D$ be a nonempty subset of $X.$ A real-valued function $g:X\longrightarrow \mathbb{R}$ is called\\
1) $D$-monotone if $x_2\in x_1+D$ implies $g(x_1)\leq g(x_2).$\\
2) strictly $D$-monotone, if $ x_2 \in x_1 + D \setminus \{0\} $
implies $ g(x_1)  <  g(x_2). $
\end{definition}

\begin{theorem}\label{sep3}
Suppose that the following conditions are satisfied:\\
$\Box$ $C$ is a proper convex and algebraic open subset of $X$;\\
$\Box$ There exists $k\in X$ such that $vcl(C)-\alpha k\subseteq
vcl(C)$ for each $\alpha>0;$\\
$\Box$ $X=\bigcup\{vcl(C)+\alpha k:~\alpha \in \mathbb{R}\}$;\\
$\Box$ $A \subset X$ is a nonempty subset of $X.$\\
Then we have:
\\
$(a)~ A \cap C = \emptyset $ if and only if there exists a convex
($\tau_c$-continuous) and onto function $ g: X \longrightarrow
\mathbb{R} $  such that
$$  g(A) \geq 0,~~  g( int_c A  )  > 0, ~~~~ g(C)  <  0,~~ g( \delta C )=0,              $$
where $ \delta  C $ stands for algebraic boundary of $C$, i.e.
$\delta  C = vcl(C) \setminus cor(C). $
\\
$(b)$ If $B\subset X$ and $\delta C-B\subset C,$ then the function
$g$ in part $(a)$ can be chosen such that it is also $B$-monotone.
\\
$(c)$ If $B\subset X$ and $\delta C-(B\setminus\{0\})\subset C,$
then one can construct function $g$ in part $(a)$ such that
it is also strictly $B$-monotone.
\\
$(d)$ If $\delta C+\delta C\subset vcl(C),$ then $g$ in part $(a)$
can be chosen such that it is subadditive.
\end{theorem}
\begin{proof} The desired results are obtained by \cite[Theorem 2.2]{ger},
because $(X,\tau_c)$ is a TVS, $int_c(C)=cor(C),$ and
$cl_c(C)=vcl(C)$.
\end{proof}

\subsection{Vector optimization}

In this subsection, we address some results in vector
optimization under real vector spaces, to show that
various important results existing in the literature can be
extended from TVSs to real vector spaces utilizing the main
results of the present paper.

Let $Z,X$ be two real vector spaces such that $X$ is
partially ordered by a nontrivial ordering convex cone $K$.
Notice that $0\notin cor(K)$ because $K\neq Y$ ($K$ is
nontrivial). Consider the following vector optimization problem:
\begin{center}
(VOP) $~~~~~~~~K-Min\: \{f(x): x\in \Omega \},$
\end{center}
where $\Omega\subseteq Z$ is a nonempty set. Here, $\Omega$ is the feasible set and $f:Z\to X$ is the objective function.

We say that a nonempty set $B\subseteq X$ is a base of the cone
$K\subseteq X$ if $0\notin B$ and for each $k\in K\backslash
\{0\}$ there are unique $b\in B$ and unique $t>0$ such that
$k=tb.$

Throughout this section, we assume that $K$ is a nontrivial convex ordering cone with a convex base $B$.
It is not difficult to show that convex ordering cone $K$ is
pointed if it has a convex base. Also, it can be shown that $0\notin vcl(B)$. Hence, the considered ordering
cone $K$ is nontrivial, convex and pointed.

Recall that a set $A\subseteq X$ is called algebraic open if
$cor(A)=A$. Also, it is vectorially closed if $vcl(A)=A$.

\begin{definition} A feasible solution $x_{0} \in \Omega$ is called an
efficient (EFF) solution of (VOP) with respect to $K$ if
$\big(f(\Omega)-f(x_{0})\big)\cap(-K)=\{0\}.$\end{definition}

The following definition extends  weak efficiency notion from
topological vector spaces to vector spaces; see
\cite{ada-2}.

\begin{definition} Assuming $core (K)\neq\emptyset$,
the feasible solution $x_{0} \in \Omega$ is called a vectorial
weakly efficient (VWEFF) solution of (VOP) with respect to $K$ if
$\big(f(\Omega)-f(x_{0})\big)\cap(-cor(K))=\emptyset.$\end{definition}

One of the solution concepts which plays an important role in
vector optimization, from both theoretical and practical points
of view, is the proper efficiency notion \cite{bot}. This concept
has been introduced to eliminate the efficient solutions with
unbounded trade offs \cite{geo}. There are different definitions
for proper efficiency in the literature; see e.g.
\cite{ben,bor,bot,geo,luc-mor,gue,hur,jah,luc,isiam} and the references
therein. Definition \ref{3.3} extends the concept of proper
efficiency given by Guerraggio and Luc \cite{luc-mor}; see \cite[Definition 2.1]{luc-mor}.

\begin{definition}\label{3.3} $x_{0} \in \Omega$ is called a
vectorial proper (VP) solution of (VOP) if there exists a convex
algebraic open set $V$ containing zero such that $cone(B+V)\neq
X$ and $x_0$ is efficient with respect to
$cone(B+V)$.\end{definition}

In the following, we extend proper efficiency notion in the
Henig's sense under vector spaces.

\begin{definition}$x_{0} \in \Omega$ is called a
vectorial Henig (VH) solution of (VOP) if there exists an ordering
pointed convex cone $C$ such that $K \backslash \{0\} \subseteq
cor(C)$ and $x_0$ is efficient with respect to
$C$.\end{definition}

Two of the most important definitions of proper efficiency have
been given by Hurwicz \cite{hur} and Benson \cite{ben}. The
following definition extends these notions under vector
spaces.

\begin{definition} \cite{ada-4} $x_{0} \in \Omega$ is called a
Hurwicz vectorial (HuV) proper efficient solution of (VOP) if
\begin{center}
$vcl\biggl (conv \biggl(cone\biggl ((f(\Omega)- f(x_{0}))\cup
K \biggr )\biggr )\biggr ) \cap (-K)=\{0\};$
\end{center}
and $x_{0}\in \Omega$ is called a Benson vectorial (BeV) proper
efficient solution of (VOP) if
$$vcl\biggl (cone\biggl (f(\Omega)- f(x_{0})+ K\biggr )
\biggr ) \cap (-K)=\{0\}.$$
\end{definition}

The following result gives a connection between the above-defined
notions. Since $(X,\tau_c)$ is a TVS, the following theorem
results from corresponding results in topological vectors spaces
(see e.g. \cite{bot,gue}) and the results of the present paper.

\begin{theorem}
(i) Each BeV proper efficient solution is
an efficient solution.\\
(ii) Each HuV
proper efficient solution is a BeV proper efficient solution. The converse holds under the convexity of $f(\Omega)+K$.\\
(iii) Each VP solution is a VH solution.\\
(iv) Each VH efficient solution is a BeV proper efficient
solution.\end{theorem}
\begin{proof} Part (i) is clear. For each $x_0\in \Omega$, the inclusion
$$cone\biggl(f(\Omega)- f(x_{0})\biggl) \subseteq conv \biggl(cone\biggl ((f(\Omega)- f(x_{0}))\cup
K \biggr )\biggr )$$
always holds and two sets are equal if $f(\Omega)+K$ is convex. This proves part (ii).

To prove part (iii), let $x_0$ be a VP solution. Then there exists a convex set
$V$ containing zero such that $cor(V)=V$ and $cone(B+V)\neq X$.
Furthermore, $x_0$ is efficient with respect to $cone(B+V)$.
Setting $C:=cone(B+V)$, the set $C$ is a convex cone. Considering
nonzero $k\in K$, there exist scalar $\lambda>0$ and vector $b\in
B$ such that $k=\lambda b$. Since $V$ is convex and $0\in
V=cor(V)$, for each $y\in X$ there exists $t^{'}>0$ such that
$\frac{t}{\lambda}y\in V$ for each $t\in [0,t^{'}]$. Hence,
$$k=\lambda b=\lambda(b+0)\in cone (B+V)=C$$
and
$$k+ty=\lambda(b+\frac{t}{\lambda}y)\in cone(B+V)=C,~~\forall t\in [0,t^{'}].$$
These imply $k\in cor(C)$. Therefore, $K\backslash \{0\}\subseteq
cor(C)$ and $x_0$ is efficient with respect to $C$. Hence, $x_0$
is a VH solution.

Part (iv) results from \cite[Proposition 2.4.11]{bot},
Theorem \ref{int}, and $vcl(.)\subseteq cl_c(.)$.
\end{proof}


Now, we use some scalarization problems. Given the convex base
$B$ of $K$, we define $$K^{ss}:=\{l\in X^{'}: \inf_{b\in B}
\langle l,b\rangle >0\}.$$ This set has been studied in some
papers including \cite{luc-mor}. The following lemma provides a sufficient condition for nonemptiness of $K^{ss}$.

\begin{lemma}
$K^{ss}\neq \emptyset$ when $K$ has a convex and relatively solid base.
\end{lemma}

\begin{proof} Let $B$ be a convex and relatively solid base of $K$. Then $ 0 \notin vcl(B) $, and thus  by Theorem \ref{sep2},  there is a nonzero linear functional $l \in X^{'}$ and  $\alpha  >  0$ such that $ \langle l,0 \rangle  <   \alpha  <  \langle l,x  \rangle $ for  every  $x \in vcl(B) $. Therefore, $\inf_{b\in B}  \langle l,b\rangle   \geq \alpha >0  $ , which means $ l \in K^{ss}.$\end{proof}

 Now considering a
linear functional $l$ belong to $K^{+},~K^{+s}$, or $K^{ss}$, we
define the following auxiliary scalarization problem:

\begin{equation}\label{auxi}\min_{x\in \Omega} \langle l,f(x)\rangle.\end{equation}

The set of optimal solutions of Problem (\ref{auxi}) is denoted
by $O_{l}$. In fact, Problem (\ref{auxi}) comes from the weighted
sum scalarization method which is a popular technique in
multiobjective optimization \cite{jah}. Now, we define some
optimal sets as follows (see also \cite{bot,luc-mor}).
$$O^+:=\{x_0\in \Omega: \exists l\in K^{+}\textmd{ s.t. } x_0\in O_{l}\},$$
$$O^s:=\{x_0\in \Omega: \exists l\in K^{+s}\textmd{ s.t. } x_0\in O_{l}\},$$
$$O^{ss}:=\{x_0\in \Omega: \exists l\in K^{ss}\textmd{ s.t. } x_0\in O_{l}\}.$$
It is clear that $O^{ss}\subseteq O^s\subseteq O^+$ and the members of $O^s$ are efficient with respect to $K$.
The following result provides a sufficient condition for VH solutions utilizing $O^s$.

\begin{theorem}\label{35}If $x_{0} \in O^s$, then $x_0$ is a VH solution.\end{theorem}
\begin{proof} Since $(X,\tau_c)$ is a TVS, this theorem results from
\cite[Proposition 2.4.15]{bot} and Theorems \ref{int} and
\ref{int1} in the present paper.
\end{proof}

Now, we state a theorem which has been proved by Adan and Novo
\cite{ada-3}. This theorem will be useful for getting an
important corollary.
\begin{theorem}\label{36} Let $K$ be solid and vectorially closed. Let $x_{0}\in \Omega$ and
$vcl\bigg(cone\big(f(\Omega)-f(x_{0})\big)+K\bigg)$ be convex. If
$x_{0}$ is a BeV solution to (VOP), then $x_{0}\in
O^s$.\end{theorem}

Theorems \ref{35} and \ref{36} lead to the following corollary
which completes the relationship between BeV solutions and VH
ones.

\begin{corollary} Suppose that $K$ is solid and vectorially closed. Furthermore, assume that $x_{0}\in
\Omega$ and $vcl\bigg(cone\big(f(\Omega)-f(x_{0})\big)+K\bigg)$ is convex. If $x_{0}$ is a BeV solution, then $x_{0}$ is a VH solution.\end{corollary}

\begin{proposition} (i) If $x_{0} \in O^{ss}$, then $x_0$ is a VP solution.\vspace{2mm}\\
(ii) If $x_{0} \in \Omega$ is a VP solution, and $f(\Omega)+K$ is convex and relatively solid, then $x_0 \in O^{ss}$. \end{proposition}
\begin{proof} This proposition results from \cite[Proposition 2.1]{luc-mor}
and Theorems \ref{int}, \ref{int1}, and \ref{vcl} in the present
paper.
\end{proof}

Below a diagram is given presenting the proved relationships between (weak) efficient solutions and proper efficient points in
different senses. In this diagram, $\lq\lq$s.a.$"$ means under $\lq\lq$some assumptions" which can be seen from the corresponding theorem. \\
$$
\begin{array}{ccc}
& O^s        &\\
& \Downarrow &\\
O^{ss}\begin{array}{c}\Longrightarrow\vspace{-2mm}\\ \displaystyle\Longleftarrow_{s.a.} \end{array} VP\Longrightarrow& VH&\begin{array}{c}\Longrightarrow\vspace{-2mm}\\ \displaystyle\Longleftarrow_{s.a.} \end{array}BeV \Longrightarrow EFF \Longrightarrow WEFF \\
&\hspace{2mm}\Uparrow \Downarrow_{s.a.}&\\
&\hspace{2mm}HuV&
\end{array}
$$

Now, we contiuou with more efficiency notions. These notions have
been investigated by Borwein and Zhuang \cite{bz} and Guerraggio
and Luc \cite{luc-mor} under normed and topological vector
spaces. We start with extending these definitions from locally
convex TVSs to real vector spaces.

\begin{definition} $x_{0} \in \Omega$ is called\\
(i)~a supper efficient solution, if for each convex algebraic
open set $V$ containing zero, there exists a convex algebraic
open set $U$ containing zero such that
$$vcl\bigg(cone\bigg(f(x_0)-f(\Omega)\bigg)\bigg)\cap
\bigg(K+U\bigg)\subseteq V.$$
(ii)~a strictly efficient solution, if there exists a convex
algebraic open set $V$ containing zero such that
$$vcl\bigg(cone\bigg(f(x_0)-f(\Omega)\bigg)\bigg)\cap \bigg(B+V\bigg)=\emptyset.$$
(iii)~a strongly efficient solution, if for each algebraic
continuous $x^*\in X^{'}$ there are convex algebraic open sets
$U,~V$ containing zero, such that $\langle x^*,.\rangle$ is
bounded on
$$cone\bigg(f(x_0)-f(\Omega)\bigg)\cap \bigg(U+cone(B+V)\bigg).$$
\end{definition}

The following result gives some connections between these different efficiency concepts
and also with the concepts studied so far.

\begin{theorem} Let $B$ be $\tau_c-$bounded, $x_{0} \in \Omega$ and $cone\bigg(f(x_0)-f(\Omega)\bigg)$ be relatively solid and convex. The following assertions are equivalent:\\
(i) $x_0$ is a supper efficient solution;\\
(ii) $x_0$ is a strictly efficient solution;\\
(iii) $x_0$ is a strongly efficient solution;\\
(iv) $x_0$ is a VP solution.
\end{theorem}
\begin{proof} This theorem results from \cite[Proposition 2.2]{luc-mor} and
Theorems \ref{int}, \ref{int1}, and \ref{vcl} in the present
paper.
\end{proof}

As a result of Theorem \ref{sep3}, the following theorem gives a characterization of weakly efficient solutions in nonconvex vector optimization by means of convex and (or) sublinear functions. The following theorem is, in fact, a direct generalization of \cite[Corollary 3.1]{ger} to vector spaces without topology.

\begin{theorem}
Let $K$ be convex and solid and $x_{0} \in \Omega$. Then\\
1. $ x_0 $ is a weakly efficient solution of (VOP) if and only if there exists a convex onto function $ g: X \longrightarrow \mathbb{R} $ which is strictly $ cor(K)-$monotone and
$$    g(f (x_0) ) =0, \quad g( f( \Omega ) ) \geq 0 , \quad g ( int_c f( \Omega ) )  >  0,                  $$
$$  g( f(x_0)  -  K  )  \leq  0, \quad  g( f(x_0)  - cor(K) )   <  0, \quad  g( f(x_0)  -  \delta K ) = 0.          $$
If $ f(x_0) = 0,$ then $g$ can be chosen such that it is subadditive. \\

2. $ x_0 $ is a weakly efficient solution of (VOP) if and only if there exists a sublinear onto function $ g: X \longrightarrow \mathbb{R}$ which is strictly $ cor(K)-$monotone and
$$   g( f( \Omega )  - f(x_0)  )  \geq  0, \quad g( int_c f( \Omega )  - f(x_0)  )  >  0 , \quad g(-K) < 0,    $$
$$   g( - \delta K ) = 0, \quad g( - cor (K) )   <  0 , \quad   g(K) \geq 0, \quad  g(cor K )   >  0.       $$
\end{theorem}
\begin{proof} The desired results are obtained by \cite[Corollary 3.1]{ger},
because $(X,\tau_c)$ is a TVS, $int_c(K)=cor(K),$ and
$cl_c(K)=vcl(K).$\end{proof}

\begin{remark}
The results of this section are some selected important issues which can be easily extended from TVSs to real vector
spaces by the main results of the present paper. Such extension can
be done for many other results in Optimization and Convex Analysis.
\end{remark}


\end{document}